\documentclass[12pt,english]{amsart}

\input xy
\xyoption{all}

\usepackage[latin1]{inputenc}
\usepackage{amsfonts,amsmath,amssymb,amsthm}
\usepackage{enumerate}

\newcommand{\Pp}{\mathbb{P}}
\newcommand{\Rr}{\mathbb{R}}
\newcommand{\Nn}{\mathbb{N}}
\newcommand{\Zz}{\mathbb{Z}}
\newcommand{\Qq}{\mathbb{Q}} 
\newcommand{\Ff}{\mathbb{F}}

\makeatletter
\def\Ddots{\mathinner{\mkern1mu\raise\p@
\vbox{\kern7\p@\hbox{.}}\mkern2mu
\raise4\p@\hbox{.}\mkern2mu\raise7\p@\hbox{.}\mkern1mu}}
\makeatother

{\theoremstyle{plain}
\newtheorem{theorem}{Theorem}[section]    

\newtheorem{addenda}[theorem]{Addendum}
\newtheorem{twisting lemma}[theorem]{Twisting lemma}

\newtheorem{proposition}[theorem]{Proposition}  
\newtheorem{corollary}[theorem]{Corollary}   

}
{\theoremstyle{remark}
\newtheorem{definition}[theorem]{Definition}      
\newtheorem{remark}[theorem]{Remark}   
\newtheorem{example}[theorem]{Example}

}

\font\tenrm=cmr10

\font\sevenrm=cmr10 scaled 700

\font\sevensl=cmsl10 scaled 700

\def\sep{{\scriptsize\hbox{\rm sep}}}
\def\spe{{0}}

\def\ur{{\scriptsize\hbox{\rm ur}}}

\def\tr{{\hbox{\sevenrm tr}}}
\def\tp{{\hbox{\sevenrm t\sevensl p}}}

\def\Gabs{\hbox{\rm G}}

\def\Gal{\hbox{\rm Gal}}

\def\ker{\hbox{\rm ker}}

\def\cm{\hbox{\hbox{\rm C}\kern-5pt{\raise 1pt\hbox{$|$}}}}

\def\lhfl#1#2{\smash{\mathop{\hbox to 12mm{\leftarrowfill}}
\limits^{#1}_{#2}}}

\def\rhfl#1#2{\smash{\mathop{\hbox to 12mm{\rightarrowfill}}
\limits^{#1}_{#2}}}

\def\build#1_#2^#3{\mathrel{
\mathop{\kern 0pt#1}\limits_{#2}^{#3}}}

\def\htrait#1#2{\smash{\mathop{\hbox to 12mm{\hrulefill}}
\limits^{#1}_{#2}}}

\def\limind{\mathop{\oalign{lim\cr
\hidewidth$\longrightarrow$\hidewidth\cr}}}

\def\limproj{\mathop{\oalign{lim\cr
\hidewidth$\longleftarrow$\hidewidth\cr}}}

\begin{document}

\title[Tchebotarev theorems for function fields]{Tchebotarev theorems for function fields}

\author{Sara Checcoli}

\author{Pierre D\`ebes}

\email{sara.checcoli@gmail.com}

\email{Pierre.Debes@univ-lille1.fr}

\address{Mathematisches Institut Universit\"at Basel, Rheinsprung 21,
CH-4051 Basel, Switzerland}


\address{Laboratoire Paul Painlev\'e, Math\'ematiques, Universit\'e 
Lille 1, 59655 Villeneuve d'Ascq Cedex, France}



\date{\today}

\begin{abstract} We prove Tchebotarev type theorems for function field extensions over various base fields: 
number fields, finite fields, $p$-adic fields, PAC fields, etc. The Tchebotarev conclusion -- existence of appropriate cyclic residue extensions -- also compares to the Hilbert specialization property. It is more local but holds in more situations and extends  to infinite extensions. For a function field extension satisfying the Tchebotarev conclusion, the exponent of the Galois group is bounded by the l.c.m. of the local specialization degrees. Further local-global questions arise for which we provide answers, examples and counter-examples.
\end{abstract}


\maketitle



\section{Introduction}

The central theme is the specialization of algebraic function field extensions, 
finite or infinite, and our main results are Tchebotarev type statements, for certain base fields (possibly infinite). 

Fix a field $K$,  a smooth projective and geometrically integral $K$-variety $B$ (typically $B=\Pp^1_K$) and a Galois extension $F/K(B)$ of group $G$. For every overfield $k\supset K$ and each point $t_0\in B(k)$, there is a notion of $k$-specialization of  $F/K(B)$ at $t_0$ (\S \ref{ssec:globaldata}); it is a Galois extension $F_{t_0}$ of $k$ of group contained in $G$, well-defined up to conjugation by elements of $G$. For example, if $B=\Pp^1_K$ and $F/K(B)$ is finite and given by an irreducible polynomial $P(T,Y)\in K(T)[Y]$, then for all but finitely many $t_0\in \Pp^1(k)$, the $k$-specialization $F_{t_0}/k$ is the extension of $k$ associated with the polynomial $P(t_0,Y)\in k[Y]$. 

The leading question is to compare the Galois groups of the specializations with the ``generic'' Galois group $G$. The Hilbert specialization property is that the ``special'' groups equal $G$ for ``many'' specializations over $k=K$. Standard situations with the Hilbert property have the base field $K$ {\it hilbertian} ({\it e.g.} a number field), the base variety $B$  
{\it $K$-rational} ({\it e.g.} $B=\Pp^1$) and concerns {\it finite} extensions $F/K(B)$.

We introduce another specialization property, which we call the {\it Tchebotarev existence property} (definition \ref{def:tchebotarev}). 
For finite extensions, it is a function field analog of the existence part in the Tchebotarev density theorem for number field extensions. Namely
we request that every conjugacy class of cyclic subgroups of $G$ be the {\it Frobenius class} of some {\it suitable} specialization $F_{t_0}/k$ of $F/K(B)$ in some local field over $K$. By ``suitable'' we mean  ``unramified and cyclic'' and ``the Frobenius class of $F_{t_0}/k$ ''  is ``the conjugacy class of $\Gal(F_{t_0}/k$)''.  Our property is weaker than the Hilbert property in the sense that it only preserves 
the ``local'' structures, but it allows more general base fields and base varieties and still 
 encapsulates a good part of the Hilbert property; and it is also defined for infinite extensions.

A main feature of an extension $F/K(B)$ with the Tchebotarev property is that it has suitable specializations with any prescribed cyclic 
subgroup of $G$ as Galois group; see \S \ref{sec:tchebotarev-property} for precise statements. We have this practical consequence (proposition \ref{prop:exposant}): 
\vskip 1,5mm

\noindent
if  \hskip 9mm {\rm (Ub-loc-d)} {\it the local degrees of $F/K(B)$ are uniformly bounded},
\vskip 1mm 

\noindent
where by local degrees we mean the degrees of all the ``suitable local specializations'' of $F/K(B)$,
\vskip 1,5mm

\noindent
then  \hskip 3mm {\rm (exp-f)}\hskip 2,5mm {\it the exponent of $G$ is finite.}
\vskip 2mm

\noindent
Furthermore the converse {\rm (exp-f)} $\Rightarrow$ {\rm (Ub-loc-d)} holds too under some standard assumptions on $K$. Finally if $G$ is abelian, {\rm (Ub-loc-d)} implies not only {\rm (exp-f)} but the following stronger condition: 
\vskip 1,5mm 

\noindent
\hskip 13mm {\it  $F\subset K(B)^{(d)}$ for some integer $d\geq 1$.}
\vskip 2mm

\noindent 
where given an integer $d\geq 1$, $K(B)^{(d)}$ denotes the compositum of all finite extensions of $K(B)$ of degree $\leq d$ (corollary \ref{cor:abelian}).

These results were established by the first author and U.~Zannier in the situation $F/K(B)$ is a number field extension $F/K$ \cite{ChZa}, \cite{Ch}. 
It turns out that the core of their arguments is the Tchebotarev property that we have identified. \S \ref{sec:tchebotarev-property} offers a formal set-up around the property and its consequences which includes both the original number field and the new function field situations.

Then in \S \ref{sec:situations} we prove our ``Tchebotarev theorems for function fields'', which provide concrete situations where the property holds.
Theorem \ref{thm:tchebotarev} and corollary \ref{cor:examples} show that a finite regular extension $F/K(T)$ (here $B=\Pp^1$ for simplicity) always has the Tchebotarev existence property if $K$ is a number field or a finite field or a PAC field\footnote{definition recalled in \S \ref{ssec:more-concrete-ex} (a).} with cyclic extensions of any degree or a rational function field $\kappa(x)$ with $\kappa$ a finite field of prime-to-$|G|$ order, etc. With some extra good reduction condition on $F/K(B)$, the property is also shown to hold if $K$ is a $p$-adic field or a formal 
Laurent series field with coefficients in a finite field, etc. To our knowledge only the finite field case was covered in the literature.

In \S \ref{ssec:tchebotarev-hilbert}, we compare our Tchebotarev property with the classical Hilbert specialization property.
The situation is clear for PAC fields for which both properties correspond to well-identified properties of the absolute Galois group of the base field $K$ (proposition \ref{prop:compare-for-PAC}). The general situation is more complex; some of the PAC conclusions still hold, others do not, and some are unclear (\S \ref{ssec:general-situation}). Still we prove that the Hilbert property is somehow squeezed between two variants of the Tchebotarev property (proposition \ref{prop:strongTch-impliesRGH}).

%

\S 5 is devoted to another natural question about the above conditions {\rm (Ub-loc-d)} and ($F\subset K(B)^{(d)}$ for some $d$), which is whether the former implies the latter in general, {\it i.e.}, without assuming $G$ abelian. The answer is ``No''; counter-examples are given in the context of number field extensions
in the Checcoli-Zannier papers. We construct other counter-examples in the situation where $\dim(B)>0$. 
One of them is re-used in a remark on a geometric analog of the Bogomolov property (definition recalled in \S \ref{sec:bogomolov}).

\vskip 2mm

\noindent
{\bf Acknowledgments.} We are grateful to Lior Bary-Soroker for his help with \S \ref{ssec:tchebotarev-hilbert}. 
Our thanks also go to Moshe Jarden and Lorenzo Ramero for their interest in the paper and several valuable references.

\section{The Tchebotarev existence property} \label{sec:tchebotarev-property}

We define the Tchebotarev existence property for finite and infinite extensions and investigate 
 its implications.

\subsection{Preliminaries}

Given a field $k$, we fix an algebraic closure $\overline k$ and denote the separable closure of $k$ in $\overline k$ by $k^\sep$ and its absolute Galois group by $\Gabs_k$. 

\subsubsection{Local fields}
Given a field $K$, what we call a {\it local field} over $K$ 
is a finite extension $k_v$ of some completion $K_v$ of $K$ for some discrete valuation $v$ on $K$. The field $k_v$ is complete with respect to the unique prolongation of $v$ to $k_v$, which we still denote by $v$.

Fields $K$ will be given with a set ${\mathcal M}$ of finite {\it places} of $K$ ({\it i.e.} of equivalence classes of discrete valuations on $K$) called a {\it localization set} of $K$. A local field $k_v$ over 
$K$ with $v\in {\mathcal M}$ is called a ${\mathcal M}$-local field over $K$.
When the context is clear, we will drop the reference to ${\mathcal M}$.

Here are some typical examples. 

\begin{example}
(a) A {\it complete valued field} $K_v$ for a non trivial discrete valuation $v$ will be implicitly given with the  localization set ${\mathcal M}=\{ v\}$. The ${\mathcal M}$-local fields over $K_v$ are $K_v$ and its finite extensions. 
\vskip 1,5mm

\noindent
(b) A {\it number field} $K$ will be implicitly given with the  localization set ${\mathcal M}$ consisting of all the finite places of $K$. The ${\mathcal M}$-local fields over $K$ are the non-archimedean completions of $K$ and its finite extensions. 
\vskip 1,5mm

\noindent
(c) If $\kappa$ is a field and $x$ an indeterminate, the {\it rational function field} $\kappa (x)$ will be implicitly given with the localization set ${\mathcal M}$ consisting of all the $(x-x_0)$-adic valuations where $x_0$ ranges over $\Pp^1(\kappa)$ (with the usual convention that $x-\infty = 1/x$). The ${\mathcal M}$-local fields over $\kappa(x)$ are the fields $\kappa((x-x_0))$ of formal Laurent series in $x-x_0$ with coefficients in $\kappa$ and their finite extensions ($x_0\in \Pp^1(\kappa)$).
\vskip 1,5mm

\noindent
(d) A {\it field} $K$, without any specification, will be implicitly given with the localization set ${\mathcal M}$ consisting of the sole trivial discrete valuation, denoted $0$.  The ${\mathcal M}$-local fields over $K$ are $K$ and its finite extensions. 
\end{example}

\subsubsection{Local specializations and Frobenius subgroups} \label{ssec:globaldata}
Suppose given a base field $K$, a smooth projective and geometrically integral $K$-variety $B$ 
and a Galois extension $F/K(B)$ 
of  Galois group $G$.

The following notions are classical when the extension $F/K(B)$ is finite and extend naturally to infinite extensions by writing $F/K(B)$ as the union of an increasing sequence of finite Galois extensions.

Given a point 
$t_0\in B(K)$, we denote by $F_{t_0}/K$ the {\it specialization of $F/K(B)$ 
at $t_0$}: if ${\rm Spec}(A) \subset B$ is some affine neighborhood of $t_0$, $A^\prime_F$ the integral closure of $A$ in $F$, then $F_{t_0}/K$ is the residue extension of the integral extension $A^\prime_F/A$ at some prime ideal above the maximal ideal corresponding to $t_0$ in ${\rm Spec}(A)$. It is a normal extension well-defined up to conjugation by elements of $G$.

\begin{definition}
Given an overfield $k$ of $K$ and 
$t_0\in B(k)$, the extension  $(Fk)_{t_0}/k$ is called a \emph{$k$-specialization of $F/K(B)$}. If $k_v$ is a local field over $K$, points $t_0 \in B(k_v)$ are
called {\it local points} of $B$, the associated $k_v$-specializations $(Fk_v)_{t_0}/k_v$ {\it local specializations} and the degrees $[(Fk_v)_{t_0}:k_v]$ {\it local degrees} of $F/K(B)$.
\end{definition}

Local degrees are to be understood as supernatural numbers \cite[\S 22.8]{FrJa} if $F/K(B)$ is infinite. 

Denote the branch locus of $F/K(B)$ by $D$, {\it i.e.}, the formal sum of all hypersurfaces of $B\otimes_K K^\sep$ such that the associated discrete valuations are ramified in the field extension $F K^\sep/K^\sep(B)$.  If the extension $F/K(B)$ is finite, $D$ is an effective divisor; in general $D$ is an inductive limit of effective divisors.

\begin{definition}
Given a local field $k_v$ over $K$ and a local point $t_0\in B(k_v)\setminus D$, the Galois group $\Gal((Fk_v)_{t_0}/k_v)$ is called the {\it Frobenius subgroup} of $F/K(B)$ at $t_0$ over $k_v$. The local point $t_0\in B(k_v)\setminus D$ is said to be {\it $k_v$-unramified} for the extension $F/K(B)$ if the associated $k_v$-specialization $(Fk_v)_{t_0}/k_v$ is unramified\footnote{When $v$ is the trivial valuation, this condition is vacuous as all finite extensions of $k_v$ are unramified.}. 
 \end{definition}

The Frobenius subgroup is a subgroup of $G$ well-defined up to conjugation by elements of $G$. Its order is the local degree $[(Fk_v)_{t_0}:k_v]$. We use the phrase {\it unramified local degree} for this degree when $t_0$ is $k_v$-unramified for $F/K(B)$.

%

\subsection{The Tchebotarev existence property} \label{sec:tchebotarev}

\subsubsection{Finite extensions} \label{ssec:finite-extensions}

\begin{definition} \label{def:tchebotarev}
(a) If $K$ is given with a localization set ${\mathcal M}$, a finite Galois extension $F/K(B)$ of group $G$ is said to have the {\it Tchebotarev existence property with respect to ${\mathcal M}$} if for every element $g\in G$, there exists a ${\mathcal M}$-local field $k_v$ over $K$ and a local point $t_0\in B(k_v)\setminus D$,  $k_v$-unramified  for $F/K(B)$, such that the Frobenius subgroup of $F/K(B)$ at $t_0$ over $k_v$ is cyclic and conjugate to the subgroup $\langle g\rangle\subset G$. 
\vskip 0,5mm

\noindent
(b) We say further that $F/K(B)$ has the {\it strict  Tchebotarev existence property} if in addition to the above, the ${\mathcal M}$-local fields $k_v$ can be taken to be completions $K_v$ of $K$ ({\it i.e.}, no finite extension is necessary). 
\end{definition}

\begin{remark}\label{rem:unr-loc-deg-bounded} If $K$ is a number field or if $K=\kappa(x)$ with $\Gabs_\kappa$ pro-cyclic, the Frobenius subgroups of $F/K(B)$ at local points $t_0\in B(k_v)\setminus D$, $k_v$-unramified for $F/K(B)$, are automatically cyclic as quotients of the pro-cyclic group $\Gal(k_v^{\ur}/k_v)$ (with $k_v^{\ur}$ the unramified closure of $k_v$). 
\end{remark}

Definition \ref{def:tchebotarev} is modelled upon the situation of number field extensions 
$F/K$. It is in fact a generalization: take $B={\rm Spec}(K)$; for every finite place of $K$, there is 
only one point in $B(K_v)={\rm Spec}(K_v)$ and the corresponding local specialization of $F/K$ is the $v$-completion of $F/K$. From the classical Tchebotarev density theorem, Galois extensions of number fields indeed have the strict Tchebotarev existence
property.\footnote{As pointed out by M. Jarden, the weaker density property proved by Frobenius 
({\it e.g.} \cite[p.134]{janusz}), where a cyclic subgroup instead of a specific element of the Galois group
is given is sufficient to prove our property for number field extensions.}
In this paper we will be more interested in function field extensions $F/K(B)$ with $\dim(B)>0$. 
Concrete situations where the {Tchebotarev existence property} is satisfied are given in \S \ref{sec:situations}.

\subsubsection{Infinite extensions} \label{ssec:tchebotarev-infinite} 
Definition \ref{def:tchebotarev} extends to infinite extensions. 

\begin{definition} A Galois extension $F/K(B)$ (possibly infinite) 
is said to have the {\it Tchebotarev existence property} (w.r.t. a localization set ${\mathcal M}$ of $K$) if $F/K(B)$ is the union of an increasing sequence of finite Galois extensions $F_n/K(B)$ 
that all have the 
Tchebotarev existence property (w.r.t. ${\mathcal M}$); and similarly for the {\it strict Tchebotarev existence property}. 
\end{definition}

This definition does not depend on the choice of the increasing sequence $(F_n/K(B))_{n\geq 1}$  such that $\bigcup_{n\geq 1} F_n = F$. This follows from the fact (left as an exercise) that given two finite Galois extensions $E/K(B)$ and $E^\prime/K(B)$ such that $E^\prime\supset E$, if $E^\prime/K(B)$ has the Tchebotarev existence property (strict or not), then so does $E/K(B)$. 

\subsection{A local-global conclusion for infinite extensions} \label{ssec:loc-glob-conclusion}
An immediate consequence of the Tchebotarev existence property is that for a finite Galois extension 
$F/K(B)$ of group $G$,
\vskip 1,5mm

\noindent
{\rm (*)} {\it the orders of elements of $G$ are exactly the unramified ${\mathcal M}$-local degrees of $F/K(B)$ corresponding to cyclic Frobenius subgroups.}
\vskip 1,5mm

\noindent
In particular the exponent of $G$ is the {l.c.m.} of these local degrees.  
Proposition \ref{prop:exposant} below shows that conclusion (*) extends in some form to infinite extensions.
The following definitions will be used.

\begin{definition} \label{def:standard}
A localization set ${\mathcal M}$ of a field $K$ is said to be {\it standard} if the local fields $k_v$ are perfect and the absolute Galois groups $\Gabs_{k_v}$ are of uniformly bounded rank {\rm (}$v\in {\mathcal M}${\rm )}.

\end{definition}
\vskip -1mm

This holds in particular in the following situations: $K$ is a number field, a $p$-adic field, 
a perfect field with absolute Galois group of finite rank ({\it e.g.} a finite field), 
a field $K=\kappa(x)$ or $K=\kappa((x))$ with $\kappa$ of characteristic $0$ and with absolute Galois group $\Gabs_\kappa$ of finite rank, etc.
\vskip 1mm

We also say that a family $(d_v)_{v}$ of positive integers indexed by $v$ 
is {\it uniformly bounded} if there is a constant $\delta$ depending on $F/K(B)$ but not on $v$ such that all integers $d_v$ are $\leq \delta$.


\begin{proposition} \label{prop:exposant} Let 
 $F/K(B)$ be a Galois extension (possibly infinite)
with Galois group $G$ and with the Tchebotarev existence property. Then 
\vskip 2mm

\noindent
if  \hskip 4mm {\rm (Ub-loc-d)} the ${\mathcal M}$-local degrees of $F/K(B)$ are uniformly bounded, 
\vskip 2mm

\noindent
then  \hskip 1mm {\rm (exp-f)}\hskip 2,5mm the exponent of $G$ is finite.
\vskip 2mm

\noindent
Furthermore the converse {\rm (exp-f)} $\Rightarrow$ {\rm (Ub-loc-d)} holds too if the localization set ${\mathcal M}$ is {\it standard}  (independently of the Tchebotarev property).
\end{proposition}

The special case of proposition \ref{prop:exposant} for which $F/K(B)$ is a Galois extension $F/K$ of number fields was 
proved in \cite{ChZa} and \cite{Ch}.

\begin{proof}[Proof of proposition \ref{prop:exposant}] Write $F/K(B)$ as an increasing union of finite Galois extensions $F_n/K(B)$ ($n\geq 1$). Let $g\in G$. For each $n\geq 1$, let $g_n$ be the projection of $g$ onto $\Gal(F_n/K(B))$.  From statement (*), for each $n\geq 1$, the order of $g_n$ is the unramified local degree $[(F_nk_v)_{t_0}:k_v]$ for some place $v\in {\mathcal M}$ and some point $t_0 \in B(k_v)\setminus D$. In particular this order divides the local degree $[(Fk_v)_{t_0}:k_v]$. This yields the following which compares to (*) above.
\vskip 2mm

\noindent {\rm (**)} {\it the set of orders of elements of $G$ is a subset of the set of all ${\mathcal M}$-local degrees of $F/K(B)$.} 
\vskip 2mm

\noindent
Implication (Ub-loc-d) $\Rightarrow$ (exp-f) is an immediate consequence.

For the converse, we borrow an argument from \cite{Ch}. Let $k_v$ be a ${\mathcal M}$-local field over $K$ and $t_0\in B(k_v)$. Fix $n\geq 1$. Assume $k_v$ is perfect. Then $(F_nk_v)_{t_0}/k_v$ is a finite Galois extension and the local degree $[(F_nk_v)_{t_0}:k_v]$ is the order of the group $\Gal((F_nk_v)_{t_0}/k_v)$. Assume further that there is a constant $N$ depending only of $F/K(B)$ such that $\Gabs_{k_v}$ is of rank $\leq N$. Then the finite group $\Gal((F_nk_v)_{t_0}/k_v)$, a quotient of $\Gabs_{k_v}$, has a generating set with at most $N$ elements. The group $\Gal((F_nk_v)_{t_0}/k_v)$ is also of exponent $\leq \exp(G)$ (as a subgroup of $\Gal(F_n/K(B))$ which itself is a quotient of $G$). If $\exp(G)$ is finite, it follows from the Restricted Burnside's Problem solved by Zelmanov (see {\it e.g.} \cite{vaughan}) that the order of the group $\Gal((F_nk_v)_{t_0}/k_v)$ can be bounded by a constant only depending on $\exp(G)$ and $N$. 
\end{proof}


The strict variant of implication (exp-f) $\Rightarrow$ (Ub-loc-d) for which only the ${\mathcal M}$-local degrees corresponding to completions of $K$ are considered holds too if $F/K(B)$ has the strict Tchebotarev property. The proof above can easily be adjusted.

\subsection{A refined question} \label{ssec:refinedquestion}
A special situation where the exponent of $G=\Gal(F/K(B))$ is finite is when 
\vskip 2mm 

\noindent
\hskip 2mm {\it $F\subset K(B)^{(d)}$ for some integer $d\geq 1$.}
\vskip 2mm

\noindent
(Indeed the Galois group $\Gal(F/K(B))$ is then a quotient of the group
 $\Gal((K(B)^{(d)}/K(B))$,  which is of exponent $\leq d!$).
 \vskip 2mm

The question then arises as to whether {\rm (Ub-loc-d)} 
implies that
$F\subset K(B)^{(d)}$ for some $d$. 
For number fields, counter-examples were given in \cite{ChZa}, \cite{Ch}. Constructing other counter-examples with $\dim(B)>0$ was a motivation for this work.  
\S \ref{sec:counterexamples} is devoted to this subtopic.

\vskip 1mm

However the answer to the question is ``Yes'' if  the group $G$ 
is  abelian. For number field extensions this was first proved in \cite{ChZa}.

\begin{corollary}\label{cor:abelian} Let $F/K(B)$ be a Galois extension with the Tchebotarev existence property. Assume that condition {\rm (Ub-loc-d)} holds and that $G=\Gal(F/K(B)$ is abelian. 
Then $F\subset K(B)^{(d)}$ for some $d$. 
\end{corollary}

\begin{proof}
From proposition \ref{prop:exposant},  $\exp(G)$ is finite. As noted 
in \cite[prop. 2.1]{ChZa}, this implies $F\subset K(B)^{(d)}$ for some $d$ if $G$ is abelian.
\end{proof}

If $F=\overline \Qq(T^{1/\infty})$ is the field generated over $\overline \Qq(T)$ by all $d$-th roots of $T$, with $d\in \Nn^\ast$, the extension $F/\overline{\Qq}(T)$ is abelian of group $G\simeq \widehat \Zz$, it satisfies condition (Ub-loc-d) (as $\overline \Qq$ is algebraically closed)  but $F\not\subset K(B)^{(d)}$ for any $d$ (as $\widehat \Zz$ is of infinite exponent). This shows that the assumption that $F/K(B)$ has the Tchebotarev property cannot be removed in corollary \ref{cor:abelian} or in implication (Ub-loc-d) $\Rightarrow$ (exp-f).

\section{Situations with the Tchebotarev property} \label{sec:situations}

Unless otherwise specified, we assume $\dim(B)>0$ in this section. In this function field context, we will mostly consider extensions $F/K(B)$ that are regular over $K$ ({\it i.e.} $F/K(B)$ is separable and $F\cap \overline K= K$).

 \subsection{Main statements} \label{ssec:examples}

\subsubsection{Main situations} \label{ssec:main-situations}
Theorem \ref{thm:tchebotarev} below is a central result of this paper: it provides various situations where the {Tchebotarev existence property} is satisfied. The proof of theorem \ref{thm:tchebotarev}  
is given in \S \ref{sec:proof-tchebotarev}.

Statement (c) uses the notion of a {\it good place} for the extension $F/K(B)$. It is defined, in de\-fi\-ni\-tion \ref{def:good-place}, by a set of conditions which classically guarantee {\it good reduction} of $F/K(B)$ (residue characteristic prime to $|G|$, etc.). 

\begin{theorem} \label{thm:tchebotarev}
Let $K$ be a field given with a localization set ${\mathcal M}$. A finite regular Galois extension $F/K(B)$ has the {Tchebotarev existence property} 
in each of the three following situations:
\vskip 1mm

\noindent
{\rm (a)} $K$ is a field that is PAC and has cyclic extensions of any degree {\rm (}with ${\mathcal M}=\{0\}${\rm )},
\vskip 1mm

\noindent
{\rm (b)} $K$ is a finite field {\rm (}with ${\mathcal M}=\{0\}${\rm )},
\vskip 1mm

\noindent
{\rm (c)} there exists a non trivial discrete valuation $v\in {\mathcal M}$ that is good for the extension $F/K(B)$ and such that the residue field $\kappa_v$ is finite, or is PAC, perfect and has cyclic extensions of any degree. 
\end{theorem}

On the other hand, there are examples for which the Tchebotarev existence property does not hold in general: for instance if $K$ is algebraically closed or if $K=\Rr$, as then, for any regular Galois extension $F/K(B)$, all specializations are of degree $1$ or $2$.

\subsubsection{More concrete examples in situations {\rm (a)-(c)}} \label{ssec:more-concrete-ex}

\paragraph{\it Situation {\rm (a)} }
\vskip 1mm

Recall that a field $k$ is said to be PAC if every non-empty geometrically irreducible $k$-variety has a Zariski-dense set of $k$-rational points. Classical results show that in some sense PAC fields are ``abundant'' \cite[theorem 18.6.1]{FrJa} and a concrete example is the field $\Qq^\tr(\sqrt{-1})$ (with $\Qq^\tr$ the field of  totally real numbers (algebraic numbers such that all conjugates are real)). 

There are many fields as in situation (a) of theorem \ref{thm:tchebotarev}. For example, it is a classical result \cite[corollary 23.1.3]{FrJa} that 
\vskip 1mm

\noindent
(*) {\it for every projective profinite group ${\mathcal G}$, there exists a PAC field $K$ such that $\Gabs_K \simeq {\mathcal G}$}.
\vskip 1mm

\noindent
For ${\mathcal G}$ chosen so that $\widehat \Zz$ is a quotient, the field $K$ satisfies condition (a) of theorem \ref{thm:tchebotarev}. Any non-principal ultraproduct of distinct finite fields is a specific example of a perfect PAC field with absolute Galois group isomorphic to $\widehat \Zz$ \cite[proposition 7.9.1]{FrJa}. 

Examples of  subfields of  $\overline{\Qq}$ can be given. The PAC field $\Qq^\tr(\sqrt{-1})$ is one: 
indeed it is also known to hilbertian and, consequently (see proposition \ref{prop:compare-for-PAC}), its absolute Galois group is a free profinite group of countable rank. It is also known that for every integer $e\geq 1$, for almost all ${\mathbf{\sigma}}=(\sigma_1,\ldots,\sigma_e)\in \Gabs_\Qq^e$, the fixed field $\overline{\Qq}^{\mathbf{\sigma}}$ of ${\mathbf{\sigma}}$ in $\overline \Qq$ is PAC and $\Gabs_{\overline{\Qq}^\sigma}$ is isomorphic to the free profinite group $\widehat F_e$ of rank $e$ \cite[theorems 18.5.6 \& 18.6.1]{FrJa}; here ``almost all'' is to be understood as ``off a subset of measure $0$'' for the Haar measure on $\Gabs_\Qq^e$. 
We note that for such fields $\overline{\Qq}^{\mathbf{\sigma}}$, a related Tchebotarev property already
appeared in \cite{Jarden-Tchebotarev}.
\vskip 1,5mm

\paragraph{\it Situation {\rm (b)} }
\vskip 1mm
The situation ``$K$ finite'' is rather classical. 
There even exist quantitative forms of the property, similar to the Tchebotarev density property for number fields;
see \cite{Weil_hermann}, \cite{Serre-zetaL}, \cite{fried-hilbert}, \cite{ekedahl}, \cite[\S 6]{FrJa}. Our approach also leads in fact  to the quantitative forms; see \cite[\S 3.5]{DEGha2} and \cite[\S 4.2]{DeLe2}. We focus here
on the existence part which also applies to infinite fields.
\vskip 1,5mm

\paragraph{\it Situation {\rm (c)} }
\vskip 1mm
The following statement provides examples. By the phrase used in (c1) and (c2) that {\it the branch locus $D$ is  good} (over $K$), we mean that it is a sum of irreducible smooth divisors with normal crossings over $K$. This is automatic if $B$ is a curve, or if, as in (c3) and (c4), $K$ has a 
place that is good for the extension $F/K(B)$ (definition \ref{def:good-place}).

\begin{corollary} \label{cor:examples} 
A finite regular Galois extension $F/K(B)$ has the Tchebotarev existence property in each of the following situations:
\vskip 1,5mm

\noindent
{\rm (c1)} $K$ is a number field and the branch locus $D$ is good,
\vskip 1,5mm

\noindent
{\rm (c2)} $K=\kappa(x)$, ${\rm char}(\kappa)=p\hskip2pt \not| \hskip2pt |G|$, 
the branch locus $D$ is good, and 

\noindent
\hskip 4mm
{\tenrm (c2-finite)} $\kappa$ is a finite field, or

\noindent
\hskip 4mm
{\tenrm (c2-PAC)} $\kappa$ is perfect, PAC and has cyclic extensions of any degree.

\vskip 1,5mm

\noindent
{\rm (c3)} $K=K_v$ is the completion of a number field at some finite place $v$ that is good for $F/K(B)$,
\vskip 1,5mm

\noindent
{\rm (c4)} $K=\kappa((x))$, the $x$-adic valuation is good for $F/K(B)$ and 

\noindent
\hskip 4mm {\tenrm (c4-finite)} $\kappa$ is a finite field, or

\noindent
\hskip 4mm
{\tenrm (c4-PAC)} $\kappa$ is perfect, PAC and has cyclic extensions of any degree.
\end{corollary}

\begin{proof} (c3) and (c4) are  obvious special cases of theorem \ref{thm:tchebotarev} (c). This is true too for (c1) and (c2): the main point is that in these cases, the localization set ${\mathcal M}$ contains infinitely many places and that only finitely many can be bad, which is clear from definition \ref{def:good-place} (under the assumption that the branch locus $D$ is good over $K$).
\end{proof}

 \subsubsection{The strict variant} \label{ssec:strict-variant}
 
 \begin{addenda}[to theorem \ref{thm:tchebotarev}] \label{add:tchebotarev}
 The strict Tchebotarev existence pro\-per\-ty is satisfied in the number field situation {\rm (c1)} and the PAC si\-tua\-tions {\rm (a)}, {\rm (c2-PAC)}, {\rm (c4-PAC)} from theorem \ref{thm:tchebotarev} and corollary \ref{cor:examples}.
 \end{addenda}
 
 Finite fields are typical examples over which the non-strict variant holds but the strict variant does not: for example if $p$ is an odd prime, the extension $F/\Ff_p(T)$ given by the polynomial $Y^2-Y - (T^p-T)$ has trivial specializations at all points $t_0 \in \Ff_p$ and so at all unbranched points $t_0\in \Pp^1(\Ff_p)$ ($\infty$ {\it is} a branch point). A similar argument (given in \S \ref{ssec:general-situation}) shows that over $\Qq_p$ the strict variant does not hold either. However we do not know whether
the non-strict variant holds over $\Qq_p$, {\it i.e.} if the condition ``$v$ good'' can be removed in corollary \ref{cor:examples} (c3).
%


\subsubsection{Equivalence between {\rm (Ub-loc-d)} and {\rm (exp-f)}}
Proposition \ref{prop:exposant} provides general links between conditions  {\rm (Ub-loc-d)} and {\rm (exp-f)}. Combining it with theorem \ref{thm:tchebotarev} and corollary \ref{cor:examples}, we obtain the following statement,  in the case $\dim(B)>0$. 

\begin{corollary} 
For a regular Galois extension $F/K(B)$, conditions {\rm (Ub-loc-d)} and {\rm (exp-f)} are equivalent in each of the following situations:
\vskip 1mm

\noindent
{\rm (a\hskip 2pt $\sharp$)} $K$ is a PAC perfect field such that $\Gabs_K$ is of finite rank and has every cyclic group as a quotient,
\vskip 1mm

\noindent
{\rm (b)} $K$ is a finite field,
\vskip 1mm

\noindent
{\rm (c1)} $K$ is a number field and the branch locus $D$ is good,
\vskip 1mm

\noindent
{\rm (c2-PAC\hskip 2pt $\sharp$)} $K=\kappa(x)$ with $\kappa$ a PAC field of characteristic $0$ such that $\Gabs_\kappa$ is of finite rank and has every cyclic group as a quotient, and the branch locus $D$ is good,
\vskip 1mm

\noindent
{\rm (c3)} $K=K_v$ is the completion of a number field at some finite place $v$ that is good for $F/K(B)$,
\vskip 1mm

\noindent
{\rm (c4-PAC\hskip 2pt $\sharp$)} $K=\kappa((x))$ if the $x$-adic valuation is good for $F/K(B)$ and for $\kappa$ a PAC field of characteristic $0$ such that $\Gabs_\kappa$ is of finite rank and has every cyclic group as a quotient.
\end{corollary}

\begin{proof}
Each situation corresponds to the conjunction of the corresponding situation in theorem \ref{thm:tchebotarev} or corollary \ref{cor:examples} and the condition from proposition \ref{prop:exposant} that the localization set ${\mathcal M}$ is standard (definition \ref{def:standard}).
It is well-known that for $\kappa$ of characteristic $p>0$, $\Gabs_{\kappa((x))}$ is not of finite type: for example, if $\kappa$ is algebraically closed, the Galois group of $X^{p^n}-X - (1/x)$ over $\kappa((x))$ is $(\Zz/p\Zz)^n$ ($n\geq 1$).
That is why situations (c2-finite) and (c4-finite) from corollary \ref{cor:examples} do not appear here and $\kappa$ is of characteristic $0$ in {\rm (c2-PAC\hskip 2pt $\sharp$)} and {\rm (c4-PAC\hskip 2pt $\sharp$)}.
\end{proof}


%

\subsection{Proof of theorem \ref{thm:tchebotarev} and of its addendum \ref{add:tchebotarev}} \label{sec:proof-tchebotarev}
A central ingredient will be  \cite{DEGha}. 
We will notably use two statements called there {\it twisting lemma} 
and {\it local specialization result}. Both are answers to the question as to whether a Galois extension $E/k$ is a specialization of a Galois $k$-cover $f:X\rightarrow B$.

Fix a finite Galois extension $F/K(B)$, regular over $K$, with group $G$ and branch locus $D$. 
Through the function field functor, it corresponds to a regular Galois $K$-cover $f:X \rightarrow B$. We use the cover viewpoint in the proof. From the Purity of Branch Locus, $f$ is \'etale above $B\setminus D$. 



\subsubsection{Good places} \label{sssec:good reduction}  

Given a local field $k_v$ over $K$, denote the va\-lua\-tion ring by $A_v$, the valuation ideal by ${\frak p}_v$, the residue field by $\kappa_v$, assumed to be perfect, its order $|\kappa_v|$ by $q_v$ and its characteristic by $p_v$. Denote also the $k_v$-cover $f\otimes_Kk_v$ by $f_v:X_v\rightarrow B_v$.

If $B$ has an integral smooth projective model ${\mathcal B}_v$ over $A_v$, we denote by ${\mathcal F}_v: {\mathcal X}_v \rightarrow {\mathcal B}_v$ the morphism corresponding to the normalization  of ${\mathcal B}_v$ in $k_v(X)$, its special fiber by ${\mathcal F}_{v,\spe}: {\mathcal X}_{v,\spe} \rightarrow {\mathcal B}_{v,\spe}$
and the Zariski closure of $D$ in ${\mathcal B}_v$ by ${\mathcal D}_v$. 

Also recall that $f$ is said to have {\it no vertical ramification at $v$} if ${\mathcal F}_v: {\mathcal X}_v\rightarrow {\mathcal B}_v$ is unramified above ${\mathfrak p}_v$ viewed as a prime divisor of ${\mathcal B}_v$. 

\begin{definition} \label{def:good-place} A place $v$ of $K$ is said to be {\it good} for $F/K(B)$ if
\vskip 1mm

\noindent
(a) $B$ has an integral smooth projective model ${\mathcal B}_v$ over $A_v$,
\vskip 1mm

\noindent
(b) $p_v=0$ or $p_v$ does not divide the order of $G$,
\vskip 1mm

\noindent
(c) each irreducible component of ${\mathcal D}_v$ is smooth over $A_v$ and  ${\mathcal D}_v \cup  {\mathcal B}_{v,\spe}$ is a sum of irreducible regular divisors with normal crossings over $A_v$,
\vskip 1mm

\noindent
(d) there is no vertical ramification at $v$ in the cover $f$.
\end{definition}

The regular $k_v$-cover $f_v$ has then {\it good reduction at $v$}: specifically, the special fiber ${\mathcal F}_{v,\spe}: {\mathcal X}_{v,\spe} \rightarrow {\mathcal B}_{v,\spe}$ is a regular cover over the residue field $\kappa_v$ with group $G$ and branch divisor ${\mathcal D}_{v,\spe}$; this follows from classical results of Grothendieck as explained in \cite[\S \S 2.4.1-2.4.4]{DEGha}. 

In the typical situation $k_v=\Qq_p$ and ${\mathcal B} = \Pp^1_{\Zz_p}$, condition (c) amounts to the branch divisor ${\bf t}$ being \'etale at $p$, and more specifically to no two  branch points $t_i,t_j \in {\overline \Qq_p} \cup \{\infty\}$ \textit{coalescing} at $v$; and coalescing at $v$ means that $|t_i|_{\overline v} \le 1$, $|t_j|_{\overline v} \le 1$ and $|t_i-t_j|_{\overline v} < 1$, or else $|t_i|_{\overline v} \ge 1$, $|t_j|_{\overline v} \ge 1$ and $|t_i^{-1}-t_j^{-1}|_{\overline v} < 1$, where $\overline v$ is any prolongation of $v$ to ${\overline \Qq_p}$.  As to the non-vertical ramification condition (d), a practical test is this:  if an affine equation $P(t,y)=0$ of $X$ is given with $t$ corresponding to $f$ and $P\in \Zz_p[t,y]$ monic in $y$, there is no vertical ramification if the discriminant $\Delta(t)$ of $P$ with respect to $y$ is non-zero modulo $p$.

\subsubsection{Proof of theorem \ref{thm:tchebotarev} and of addendum \ref{add:tchebotarev}}
\label{ssec:proofs}
Let $g\in G$. The strategy is to construct a ${\mathcal M}$-local field $k_v$ over $K$ such that 

\noindent
(i) there exists an unramified Galois extension $E/k_v$ with Galois group isomorphic to the subgroup $\langle g\rangle \subset G$, and 

\noindent
(ii) the extension $E/k_v$ is a specialization of the extension $F\hskip 1ptk_v/k_v(B)$ at some point $t_0 \in B(k_v)\setminus D$. 

\noindent
We will conclude that the group $\Gal((F\hskip 1pt k_v)_{t_0}/k_v)$, {\it i.e.}, the Frobenius subgroup of $F/K(B)$ at $t_0$ over $k_v$, is cyclic and conjugate to $\langle g \rangle$ in $G$.

To achieve (ii) we will use the {\it twisting lemma} from \cite{DEGha}, which says the following. Let $\varphi: \Gabs_{k_v} \rightarrow \langle g\rangle$ be an epimorphism such that the fixed field $(k_v^{\sep})^{{\sevenrm ker}(\varphi)}$ is an extension $E$ of $k_v$ as in (i). Then there is a regular
$k_v$-cover $\widetilde f^{\varphi}_v: \widetilde X_v^{\varphi} \rightarrow B_v$ (with $B_v=B\otimes_K k_v$) such that 
\vskip 1,5mm

\noindent
(*) {\it condition {\rm (ii)} holds if and only if there exists a $k_v$-rational point on $\widetilde X_v^{\varphi}$ not lying above any point in the branch locus $D$.}
\vskip 1,5mm

\noindent
The cover $\widetilde f^{\varphi}_v: \widetilde X_v^{\varphi} \rightarrow B_v$ is obtained by ``twisting'' $Fk_v/k_v(B)$, viewed as a regular Galois $k_v$-cover $f_v: X_v \rightarrow B_v$, by the epimorphism $\varphi$, whence the terminology and the notation. 

The proof of (a) follows at once. Take for $v$ the trivial valuation on $K$ (for which $K_v=K$). From the assumption an extension $E/K$ as in (i) exists, and by definition of PAC fields, the set $\widetilde X_v^{\varphi}(K)$ is Zariski-dense, and so (ii) holds as well\footnote{For PAC fields, stronger results can be proved for which $\langle g\rangle$ can be replaced by any subgroup of $G$; see \cite[corollary 3.4]{DEGha2}.}. Furthermore it is the {strict Tchebotarev existence property} (and so addendum \ref{add:tchebotarev} (a)) that has been proved. 

\begin{remark} \label{rem:nonstrict_PAC}
The non-strict Tchebotarev existence property holds under a weaker condition: the argument above shows that it is sufficient that every cyclic subgroup $C$ be the Galois group of some finite extension $E_C/k_C$ with $k_C$ a finite extension of $K$.
\end{remark}

The proof of (b) goes along similar principles but with the Lang-Weil estimates replacing the PAC property. More precisely assume that $K$ is the field $\Ff_{q_0}$ with $q_0$ elements. Pick 
a suitably large integer $m$; more specifically $q=q_0^m$ should be bigger than the constant $c$ from \cite[corollary 3.5]{DEGha2}, which depends only on $G$, $B$ and $D$. Then from that result, if $d$ is the order of $g$, the extension $\Ff_{q^{d}}/\Ff_q$ is the specialization of $F\hskip 2pt \Ff_q/\Ff_q(B)$ at some point $t_0\in B(\Ff_q) \setminus D$. So the extension
$F\hskip 2pt \Ff_q/\Ff_q(B)$ satisfies conditions (i) and (ii) above for $v$ the trivial place on $K=\Ff_{q_0}$ and $k_v=\Ff_q$. We note that we have used a scalar extension (from $\Ff_{q_0}$ to $\Ff_q$) and only proved the {(non-strict) Tchebotarev property}.

The proof of (c) relies on proposition 2.2 from \cite{DEGha}, which we apply to the $k_v$-cover $f_v\otimes_{K_v}k_v$ and to the unramified homomorphism $\varphi:\Gabs_{k_v} \rightarrow \langle g \rangle \subset G$ defined as follows. If the residue field $\kappa_v$ is PAC, take $k_v=K_v$ and if it is a finite field $\Ff_{q_0}$ with $q_0$ elements, take $k_v$ equal to the unique unramified extension of $K_v$ with residue extension $\Ff_{q_0^m}/\Ff_{q_0}$ with $q=q_0^m$  bigger than the constant $c$ from \cite[lemma 2.4]{DEGha} (which is some version of the constant $c$ used above).
In both cases, denote the residue field of $k_v$ by $\tilde \kappa_v$. From the hypotheses, the field $\tilde \kappa_v$ has a Galois extension $\varepsilon_v/\tilde \kappa_v$ of group $\langle g \rangle$. Let $E_v/k_v$  be the unique unramified extension with residue extension $\varepsilon_v/\tilde \kappa_v$ and $\varphi:\Gabs_{k_v} \rightarrow \langle g \rangle$ be an epimorphism such that the fixed field $(k_v^{\sep})^{{\sevenrm ker}(\varphi)}$ is $E_v$. 

Proposition 2.2 from \cite{DEGha} has two assumptions which are labelled (good-red) and ($\kappa$-big-enough). The former is here covered by the assumption that $v$ is good for $F/K(B)$. The latter holds as well: this follows from the PAC property if $\kappa_v$ is PAC, and from \cite[lemma 2.4]{DEGha} if $\kappa_v$ is finite of order $> c$. Conclude then from \cite[proposition 2.2]{DEGha} that there exists $t_0\in B(k_v)\setminus D$ such that the specialization $(F\hskip 1pt k_v)_{t_0}/k_v$ is conjugate to $E_v/k_v$. In particular $\Gal((F\hskip 1pt k_v)_{t_0}/k_v)$ is cyclic and conjugate to $\langle g \rangle$ in $G$. Furthermore we have proved the {strict Tchebotarev property} in the case of a PAC residue field (and so addendum \ref{add:tchebotarev} (c2-PAC) and (c4-PAC)) but only the {non-strict Tchebotarev property} in the case of a finite residue field.

It remains to show addendum \ref{add:tchebotarev} (c1). That is, to prove the strict Tchebotarev property assuming that $K$ is a number field and the branch locus $D$ is good. Denote by ${\mathcal B}$ an integral projective model of $B$ over the ring $R$ of integers of $K$; ${\mathcal B}$ is smooth over the completion $R_v$ for all finite places of $K$ but in a finite subset $S_0$. Pick a place $v$ of $K$ that is good (in particular $v\notin S_0$) and has a residue field $\kappa_v$ of order bigger than the constant $C(f,{\mathcal B})$ from \cite[lemma 3.1]{DEGha}. As above, assumptions (good-red) and ($\kappa$-big-enough) from \cite[proposition 2.2]{DEGha} are guaranteed and it can be concluded that there exists $t_0\in B(k_v)\setminus D$ such that the specialization $(F\hskip 1pt K_v)_{t_0}/K_v$ is conjugate to 
the unique unramified extension $E_v/K_v$ of degree the order of $g$. $\square$

\subsection{A further example} We illustrate our method with a last situation where the residue 
fields are neither PAC nor finite. A typical example we have in mind in the statement below is this: $K$ is the field $k_0((\theta)) (x)$ with $x$ and 
$\theta$ two indeterminates and the localization set consists of all $(x-x_0)$-adic valuations with $x_0\in \Pp^1(k_0((\theta)))$.

\begin{theorem} \label{cor:further-example} 
Assume $K$ is given with a localization set ${\mathcal M}$ that contains a non trivial discrete 
valuation $v\in {\mathcal M}$ such that
\vskip 1mm

\noindent
{\rm (a)} the residue field $\kappa_v$ is a complete field for a non trivial discrete valuation $w$ 
with a residue field  ${\kappa}_{v,w}$ that is perfect, PAC and has cyclic extensions of any degree.
\vskip 1mm

\noindent
Then a  finite regular Galois extension $F/K(B)$ has the strict Tchebotarev existence property 
if $G$ is of trivial center and $B$ has an integral smooth projective $A_v$-model 
${\mathcal B}_v$ such that
\vskip 1mm

\noindent
{\rm (b)}  $v$ is good for this model of $F/K(B)$,
\vskip 1mm

\noindent
{\rm (c)} the place $w$ is good for the extension $\kappa_v({\mathcal X}_{v,0})/\kappa_v({\mathcal B}_{v,0})$ ({\it i.e.}, the function field extension of the special fiber of ${\mathcal F}_v: {\mathcal X}_v \rightarrow {\mathcal B}_v$).
\end{theorem}

For $K = k_0((\theta)) (x)$, condition (a) holds if  $k_0$ is a perfect PAC field with cyclic extensions of any degree. For all but finitely many $x_0\in \Pp^1(k_0((\theta)))$, the $(x-x_0)$-adic valuation $v_{x_0}$ is good for $F/K(B)$, {\it i.e.} condition (b) holds. The special fiber  
is a $k_0((\theta))$-cover and condition (c) requires that  the $\theta$-adic valuation on $k_0((\theta))$ be good for it.

\begin{proof}
Fix $g \in G$. The proof follows the same strategy as in \S \ref{ssec:proofs} and uses again \cite[proposition 2.2]{DEGha}, applied here to the $K_v$-cover $f_v=f \otimes_{K}K_v$ and the unramified homomorphism $\varphi:\Gabs_{K_v} \rightarrow \langle g \rangle \subset G$ defined as follows. From assumption (a), there exists a Galois extension of $\kappa_{v,w}$ of group isomorphic to $\langle g\rangle$. This extension lifts to an unramified (w.r.t. $w$) extension of $\kappa_v$ with the same group, which in turn lifts to an unramified (w.r.t. $v$) extension $E_v/K_v$ with the same group $\langle g\rangle$. Let $\varphi:\Gabs_{K_v} \rightarrow \langle g \rangle \subset G$ be an associated representation of $\Gabs_{K_v}$, {\it i.e.}, the fixed field of $\ker(\varphi)$ in $\overline{K_v}$ is $E_v$.

The $K_v$-cover $f_{v}$ satisfies condition (good-red) from \cite[proposition 2.2]{DEGha}; this is guaranteed by assumption (b).

 To check condition ($\kappa$-big-enough) from  \cite[proposition 2.2]{DEGha}, we give ourselves
what is called an {\it $A_v$-model of $(f_v\otimes_{K_v}K_v^\sep, {\mathcal F}_{v,\spe}\otimes_{\kappa_v} \overline {\kappa_v})$} in \cite{DEGha}, {\it i.e.}, a finite and flat morphism ${\mathcal F}^\prime: {\mathcal X}^\prime\rightarrow {\mathcal B}_v$ with ${\mathcal X}^\prime$ normal and such that ${\mathcal F}^\prime \otimes_{A_v} K_v$ is a $K_v$-cover that is $K_v^\sep$-isomorphic to $f_v\otimes_{K_v}K_v^\sep$ and the special fiber ${{\mathcal F}^\prime_\spe}: {\mathcal X}^\prime_\spe \rightarrow {\mathcal B}_{v,\spe}$ is a $\kappa_v$-cover that is $\overline{ \kappa_v}$-isomorphic to ${\mathcal F}_{v,\spe} \otimes_{\kappa_v} \overline{\kappa_v}$. And we have to  find $\kappa_v$-rational points on ${\mathcal X}^\prime_\spe$ not lying above any point in ${\mathcal D}_\spe\otimes_{\kappa_v} \overline{\kappa_v}$. 

Denote the valuation ring of $w$ by $A_{v,w}$. From assumption (c), the $\kappa_v$-variety ${\mathcal B}_{v,\spe}$ has an integral smooth projective model $\widetilde{\mathcal B}_0$ over $A_{v,w}$, and $w$ is good for this model of $\kappa_v({\mathcal X}_{v,0})/\kappa_v({\mathcal B}_{v,0})$.
It follows that $w$ is also good for $\kappa_v({\mathcal X}^\prime_\spe)/\kappa_v({\mathcal B}_{v,0})$. Indeed conditions (a), (b), (c) from definition \ref{def:good-place} are equivalently satisfied by
the place $w$ for either one of the two extensions. As to condition (d), we resort to a result of S.~Beckmann \cite{Beckmann} that says that non-vertical ramification is automatic under (a), (b), (c) 
if $G$ is of trivial center. It follows that $\widetilde{\mathcal F}^\prime_\spe$ has good reduction 
(at $w$). As $\kappa_{v,w}$ is PAC, there exist $\kappa_{v,w}$-rational points on the reduction (at $w$) of $\widetilde{\mathcal X}^\prime_\spe$ that are not in the branch locus of the reduction (at $w$) of $\widetilde{\mathcal F}^\prime_\spe$. Using Hensel's lemma, these points can be lifted to $\kappa_v$-rational points on ${\mathcal X}^\prime_\spe$ as desired.

Proposition 2.1 from \cite{DEGha} can then be applied to conclude that the unramified extension $E_v/K_v$, cyclic of group $\langle g\rangle$, is a $K_v$-specia\-li\-za\-tion of the extension $F/K(B)$. 
\end{proof} 

\begin{remark}
A non-strict variant of theorem \ref{cor:further-example} can be proved if the residue field $\kappa_{v,w}$ is assumed to be {finite} instead of PAC. 
The modifications are similar to those in the proof of  theorem \ref{thm:tchebotarev} (for (b)\hskip 0,5mm {\it vs.}\hskip 0,5mm (a)): the Lang-Weil estimates replace the PAC property, a finite extension of $K_v$ is needed to insure that the finite residue field $\kappa_{v,w}$ is big enough, etc.  We leave the reader adjust the proof.
\end{remark}

\section{Tchebotarev versus Hilbert}  \label{ssec:tchebotarev-hilbert}
 We compare the Tchebotarev existence property and the Hilbert specialization property. For short we say that a field $K$ given with a localization set ${\mathcal M}$ is {\it Tchebotarev} (resp. {\it strict Tchebotarev}) if every finite regular Galois extension $F/K(T)$ has the Tchebotarev (resp. {\it strict Tchebotarev}) existence property.



From \S \ref{sec:tchebotarev}, PAC fields and number fields are strict Tchebotarev, finite fields are Tchebotarev, but not strict Tchebotarev.

Recall that a finite extension $F/K(T)$ is said to have the {\it Hilbert specialization property} if it has infinitely many specializations $F_{t_0}/K$ at points $t_0\in \Pp^1(K)$ of degree equal to $[F:K(T)]$ and that a field
$K$ is called {\it hilbertian} if the Hilbert specialization property holds for every 
finite extension $F/K(T)$ and {\it RG-hilbertian} if it holds for every finite regular Galois extension $F/K(T)$.

\subsection{The PAC situation {\rm gives a first idea of these notions hierarchy}}
Recall the following definition that is used in statement (a) below: a field $K$ is $\omega$-free if every embedding problem for $\Gabs_K$ is solvable \cite[\S 27.1]{FrJa}. From a theorem of Iwasawa, if $\Gabs_K$ is of at most countable rank, $K$ is $\omega$-free if and only if  $\Gabs_K$ is isomorphic to the free profinite group $\hat F_\omega$ with countably many generators \cite[theorem 24.8.1]{FrJa}.

Conclusions (a) and (b) below are classical; see \cite[corollary 27.3.3]{FrJa} for the {\it if} part in (a), 
\cite[theorem A]{FV92} for the {\it only if} part, and \cite[theorem B]{FV92} for (b)\footnote{ \cite{FV92} assumes $K$ of characteristic $0$ and countable, but these hypotheses have been removed in subsequent works; see \cite{PopLarge} for (a) and  \cite[\S 3.3]{DeBB} for (b).}. We have included them in the statement to put the new conclusions (c) and (d) in perspective.

\begin{proposition} \label{prop:compare-for-PAC}
 Let $K$ be a PAC field given with the trivial localization set ${\mathcal M}=\{0\}$.
\vskip 1mm

\noindent
{\rm (a)} $K$ is hilbertian iff $K$ is $\omega$-free.
\vskip 1mm

\noindent
{\rm (b)} $K$ is RG-hilbertian iff every finite group is a quotient of $\Gabs_K$.
\vskip 1mm

\noindent
{\rm (c)} $K$ is strict Tchebotarev iff every cyclic group is a quotient of $\Gabs_K$.
\vskip 1mm

\noindent
{\rm (d)} $K$ is Tchebotarev iff every cyclic group $C$ is a quotient of some open subgroup $U_C$ of $\Gabs_K$.
\vskip 1mm

\noindent
In particular we have this chain of implications:
\vskip 1mm

\noindent
\centerline{{\rm hilbertian} $\Rightarrow$ {\rm RG-hilbertian} $\Rightarrow$ {\rm strict Tchebotarev} $\Rightarrow$ {\rm Tchebotarev} }
\vskip 1mm

\noindent
Furthermore none of the reverse implications holds.
\end{proposition}


\begin{proof}
The {\it if} part in (c) is theorem \ref{thm:tchebotarev} (a).
For the {\it only if} part, let $G$ be a cyclic group. Classically every cyclic group $G$ is the group of some regular Galois extension $F/K(T)$.
If $K$ is strict Tchebotarev, then a specialization $F_{t_0}/K$ of group $G$ does exist. Similar arguments lead to the non-strict variant (d) of (c) (use remark \ref{rem:nonstrict_PAC} for the {\it if} part).

Using the classical result (*) recalled in \S \ref{ssec:more-concrete-ex}, the search of counter-examples to the reverse implications can be reduced to that of projective profinite groups ${\mathcal G}$ with appropriate properties. For a counter-example to ``strict Tchebotarev $\Rightarrow$ RG-hilbertian'', take ${\mathcal G} = \widehat \Zz$ and a PAC field $K$ such that $\Gabs_K \simeq {\mathcal G}$. From statements (b) and (c), $K$ is strict Tchebotarev but is not RG-hilbertian.
For a counter-example to ``RG-hilbertian $\Rightarrow$ hilbertian'', see \cite[\S 2]{FV92}. Finally for the implication  ``Tchebotarev $\Rightarrow$ strict Tchebotarev'', we have the following counter-example, provided to us by Bary-Soroker.

Take for ${\mathcal G}$ the universal Frattini cover \cite[\S 22.6]{FrJa} of the group $\prod_{n\geq 5} A_n$ and a PAC field $K$ such that $\Gabs_K \simeq {\mathcal G}$. From \cite[lemma 22.6.3]{FrJa}, if a cyclic subgroup $C$ is a quotient of ${\mathcal G}$, then $C$ is a Frattini cover of a quotient $D$ of $\prod_{n\geq 5} A_n$. But then from \cite[lemma 25.5.3]{FrJa}, $D$ is a direct product of alternating groups $A_n$: a contradiction if $C$ is non-trivial. Conclude {\it via} statement (d) that $K$ is not strict Tchebotarev. Now as we explain below, $K$ is Tchebotarev. For every integer $m\geq 1$, the alternating group 
$A_{2m}$ is a quotient of ${\mathcal G}$. Denote by $K_{2m}/K$ the corresponding Galois extension, 
of group $A_{2m}$. If $\sigma_m\in A_{2m}$ is the product of two $m$-cycles and $k$ the fixed field of $\sigma_m$ in $K_{2m}$, then $k/K$ is finite and  $K_{2m}/k$ is Galois of group $\langle \sigma_m \rangle$. As $m$ is arbitrary, this indeed shows that every cyclic subgroup is a quotient of some 
open subgroup of $\Gabs_K$ and so {\it via} (d) that $K$ is Tchebotarev.
\end{proof}


\subsection{The general situation {\rm over non PAC fields $K$ and for not necessarily trivial localization sets $\mathcal M$ is more complex} } \label{ssec:general-situation}    We explain in this subsection what remains in general of 
the first four equivalences of proposition \ref {prop:compare-for-PAC} and in the next one how 
some implications between the various properties can still be obtained.
\vskip 0mm

Assume $K$ is an arbitrary field. 

\subsubsection{}
Implication ($\Rightarrow$) in proposition \ref{prop:compare-for-PAC} (a) does not hold: for example, $\Qq$
 is hilbertian but not $\omega$-free.\footnote{However it is conjectured that ``$K$ hilbertian'' implies that every {\it split} finite embedding problem over $K$ has a solution \cite{DeDes1} (which itself implies $K$ $\omega$-free if in addition $\Gabs_K$ is projective and countable).}  Implication ($\Rightarrow$)  in proposition \ref{prop:compare-for-PAC} (c) and (d) still holds: the argument is the same as for PAC fields (and this argument also shows that implication ($\Rightarrow$) in proposition \ref{prop:compare-for-PAC} (b) also holds if every finite 
group is the Galois group of some regular Galois extension; see also \cite[\S 3.3.2]{DeBB}).


\subsubsection{} None of the converses hold in general. For (a), see \cite[remark 2.14]{Bary-Soroker-Paran}. 
For (b) and (c), take a prime $p$ and consider the field $\Qq^{\tp}$ of all totally $p$-adic algebraic numbers. It is known that every finite 
group is a quotient of $\Gabs_{\Qq^{\tp}}$ \cite{efrat}. But if $F/\Qq^{\tp}(T)$ is the extension given by the polynomial $P(T,Y)= Y^2-Y - (pT/T^2-p)$, then for every $t_0\in \Pp^1(\Qq^{\tp})$, the polynomial $P(t_0,Y)$ is split in $\Qq^{\tp}[Y]$ \cite[example 5.2]{DeHa}. Therefore $F/\Qq^{\tp}(T)$ has no $\Qq^{\tp}$-specialization with Galois group $\Zz/2\Zz$ and so $\Qq^{\tp}$ is not strict Tchebotarev. This example also shows that $\Qq_p$ is not strict Tchebotarev and so yields another counter-example to the converse in (c).
One may think that $\Qq_p$ is not even Tchebotarev; it would then also be a counter-example to ($\Leftarrow$) in (d).

\subsection{Tchebotarev versus Hilbert: general case} 
Proposition \ref{prop:strongTch-impliesRGH} below shows that the Hilbert property is squeezed between a strong and a weak variant of the Tchebotarev property.






\begin{definition} \label{def:strongtchebotarev}
If $K$ is given with a localization set ${\mathcal M}$, a finite Galois extension $F/K(B)$ is said to have the {\it strong Tchebotarev existence property with respect to ${\mathcal M}$} if for every element $g\in G$, there exist {\it infinitely many places} $v\in {\mathcal M}$  
with corresponding 
points $t_v\in B(K_v)\setminus D$ $k_v$-unramified for $F/K(B)$ and such that the Frobenius subgroup of $F/K(B)$ at $t_0$ over $K_v$ is cyclic and conjugate to the subgroup $\langle g\rangle\subset G$.  \end{definition}

We also say that  $K$ is {\it strong Tchebotarev} if every finite regular Galois extension $F/K(T)$ has the strong Tchebotarev existence property. 

\begin{proposition} \label{prop:strongTch-impliesRGH}
Let $F/K(T)$ be a finite regular Galois extension.
\vskip 1mm

\noindent
{\rm (a)} If $F/K(T)$ has the strong Tchebotarev existence property w.r.t. a localization set ${\mathcal M}$ of $K$, then it has the Hilbert specialization property. In particular, if $K$ is {\it strong Tchebotarev}, then it is RG-hilbertian.
\vskip 1mm

\noindent
{\rm (b)} If $K$ is a countable hilbertian field then $F/K(T)$ has the Tchebotarev existence property
w.r.t. the trivial localization set ${\mathcal M}=\{0\}$.  In particular, $K$ is {\it Tchebotarev} w.r.t.  ${\mathcal M}=\{0\}$.
\end{proposition}

\begin{proof} (a) Definition \ref{def:strongtchebotarev} makes it possible to construct a family of places $(v_g)_{g\in G}$, pairwise distinct and with the property that for each $g\in G$, there exists $t_{v_g}\in \Pp^1(K_{v_g})\setminus D$ $k_v$-unramified for $F/K(T)$ and such $\Gal(F_{t_{v_g}}K_{v_g}/K_{v_g})$ is conjugate to $ \langle g\rangle$. For each $g\in G$, the set of such points $t_{v_g}$ is a $v_g$-adic subset of  $\Pp^1(K_{v_g})\setminus D$; this follows from the twisting lemma recalled in \S \ref{ssec:proofs} (*). Using the approximation Artin-Whaples theorem, the collection of points $(t_{v_g})_{g\in G}$ can be approximated by some point $t_0\in \Pp^1(K)\setminus D$ such that  $\Gal(F_{t_0}K_{v_g}/K_{v_g})$ is conjugate to $\langle g\rangle$ for each $g\in G$. As $\Gal(F_{t_0}K_{v_g}/K_{v_g})$ is a subgroup of $\Gal(F_{t_0}/K)$, conclude that $\Gal(F_{t_0}/K)$ meets each conjugacy class of $G$. By a classical lemma of Jordan \cite{Jordan}, $\Gal(F_{t_0}/K)$ is all of $G$.
\vskip 1mm

(b) The following proof is due to L.~Bary-Soroker. From \cite[theorem 18.10.2]{FrJa}, the countable hilbertian field $K$ can be embedded in some field $E$, Galois over $K$, PAC and $\omega$-free. 
From proposition \ref{prop:compare-for-PAC}, $E$ is hilbertian, and consequently is strict 
Tchebotarev w.r.t. ${\mathcal M}=\{0\}$. It readily follows that $F/K(T)$ has the Tchebotarev existence property (and that $K$ is Tchebotarev w.r.t. ${\mathcal M}=\{0\}$). Indeed given any $g\in \Gal(F/K(T)) = \Gal(F E/E(T))$, there exists $t_0\in \Pp^1(E)$ such that $\langle g\rangle = \Gal((F E)_{t_0}/E)$ and a  standard argument shows that the same is true with $E$ replaced by some finite extension $k$ of $K$. \end{proof}

The proof shows that proposition \ref{prop:strongTch-impliesRGH} (a) still holds if $F/K(T)$ is replaced by an extension $F/K(B)$ with $B$ satisfying the {\it weak approximation property} (and even {\it the weak weak approximation property} \cite[d\'efinition 3.5.6]{Serre-topics}).

\section{A question on infinite extensions} \label{sec:counterexamples}

This section is devoted to the question which arose in \S \ref{ssec:refinedquestion}.

Fix a field $K$ with a localization set ${\mathcal M}$, assumed to be standard (the definition and a list of examples are given in \S \ref{ssec:loc-glob-conclusion}). 
%
Fix also a smooth projective and geometrically integral $K$-variety $B$.

\subsection{The question and the main result} \label{ssec:the-question}
Given a Galois extension $F/K(B)$ of group $G$, the following conditions 
were introduced in \S \ref{ssec:refinedquestion}

\vskip 2mm

\noindent 
{\rm (Ub-loc-d)} {\it the ${\mathcal M}$-local degrees of $F/K(B)$ are uniformly bounded.}
\vskip 2mm

\noindent 
 {\it $F\subset K(B)^{(d)}$ for some integer $d\geq 1$.} 
\vskip 2mm

\noindent 
Under the assumption that $F/K(B)$ has the Tchebotarev existence property, we showed that if $G$ is abelian, (Ub-loc-d) implies $F\subset K(B)^{(d)}$ for some $d$ (corollary \ref{cor:abelian}) and that for $G$ arbitrary,  (Ub-loc-d) only implies that the exponent of $G$ is finite (proposition \ref{prop:exposant}). The question remains whether 
 (Ub-loc-d) implies $F\subset K(B)^{(d)}$ for some $d$ in general. We will produce several examples showing 
that it does not.

Our examples will even satisfy this stronger variant of (Ub-loc-d): 
\vskip 2mm

\noindent 
(Ub-dec-d)  {\it the ${\mathcal M}$-local decomposition degrees of $F/K(B)$ are uniformly bounded.}
\vskip 2mm

\noindent
where by {\it local decomposition degree} at some $\mathcal{M}$-local point $t_0\in B(k_v)$, we mean the order of the decomposition group of $Fk_v/k_v(B)$ at $t_0$ (while the local degree is the degree of the residue extension).

More specifically we will prove the following.

\begin{theorem}  \label{thm:3-examples}
In the following situations, there exists an infinite 
Galois extension $F/K(T)$ satisfying {\rm (Ub-dec-d)} but such that $F\subset K(B)^{(d)}$ for any integer $d$:
\vskip 1mm

\noindent
{\rm (a)} The RIGP holds over $K$ and the localization set ${\mathcal M}$ is standard.
Furthermore the constructed extension $F/K(T)$ is regular over $K$.
\vskip 1,3mm

\noindent
{\rm (b)} $K$ is a finite field and $B=\Pp^1$.
\end{theorem}

Recall that the RIGP (Regular Inverse Galois Problem) is the condition that every finite group is the Galois group of some regular Galois extension $F/K(T)$.
The RIGP is known to hold over PAC fields and complete valued fields. So such fields with a standard localization set are examples of fields as in (a). Conjecturally the RIGP holds over every field and so all fields $K$ with  a standard localization set, {\it e.g.} number fields, are other examples.


\subsection{Proof of theorem \ref{thm:3-examples}} 
We will adjust to our function field context a construction given in \cite[\S 3]{ChZa} in the context of number fields.

\subsubsection{Strategy} \label{ssec:construction}
The construction uses extra-special groups. We recall their definition and refer to \cite[\S A.20]{Doerk} for more details.

\begin{definition} Given a prime number $\ell$, a finite $\ell$-group  $E$ is said to be $\emph{extra-special}$ if its center $Z(E)$ and its commutator subgroup $E^\prime$ have both order $\ell$ (and then $Z(E)=E^\prime$).
\end{definition}


Fix two odd primes $\ell$ and $q$ such that $\ell\mid q-1$.  Then for every positive integer $m\geq 1$, is known to exist an extra-special group  of order $\ell^{2m+1}$, of exponent $\ell$ and of rank $2m$. Fix one such group $E_m$ ($m\geq 1$). Moreover there exists an irreducible $E_m$-module of dimension $\ell^m$ over the finite field $\mathbb{F}_q$. Fix such an $E_m$-module $W_m$, and finally denote by $G_m$ the semi-direct product $W_m\rtimes E_m$ ($m\geq 1$).

The following statement summarizes the strategy from \cite[\S 3]{ChZa}.

\begin{proposition} \label{c_ex_p} Assume $B$ is a curve
and for each $m\geq 1$, $G_m$ is the group of a Galois extension $F_m/K(B)$. Let $F/K(B)$ be the compositum  of all extensions $F_m/K(B)$. Then $F$ is not contained in $K(B)^{(d)}$ for any $d$ but $F/K(B)$ satisfies the {\rm tame} variant of {\rm (Ub-dec-d)} for which the decomposition degrees are requested to be uniformly bounded at all ${\mathcal M}$-local points $t_0\in B(k_v)$ that are {\rm tamely} 
branched in  $Fk_v/k_v(B)$.
\end{proposition}

\begin{proof} The proof is given in \cite{ChZa} in the case $\dim(B)=0$ and can be used in the more general case $\dim(B) \geq 0$ with almost no changes. Proposition 3.1 and proposition 3.3 of \cite{ChZa} show that 
$F$  is not contained in $K(B)^{(d)}$ for any integer $d\geq 1$ and that $G=\Gal(F/K(B))$ is of finite exponent. From proposition \ref{prop:exposant}, this implies that the local degrees of $F/K(B)$ are uniformly bounded. For each $t_0\in B(k_v)$ the local degree of $F/K(B)$ at $t_0$ is the degree of the residue field extension above 
the  point $t_0$. Thus it remains to prove that the inertia subgroups at all ${\mathcal M}$-local points $t_0\in B(k_v)$ that are tamely branched in the extension $Fk_v/k_v(B)$ are of uniformly bounded orders. By definition of ``tame branching'', these inertia subgroups are pro-cyclic subgroups of $G$, and so are of order $\leq \exp(G)$.
\end{proof}


\subsubsection{End of proof of theorem \ref{thm:3-examples}}\label{ssec:finite}
We use the construction from \S \ref{ssec:construction} with the primes $\ell, q$ distinct from $p$. Under the hypotheses of  theorem \ref{thm:3-examples}, for each $m\geq 1$, we have a Galois extension $F_m/K(T)$ of group $G_m$. This is clear in case (a) of  theorem \ref{thm:3-examples}; the extension $F_m/K(T)$ can further be taken to be regular over $K$. In case (b) for which $K$ is finite, we resort to Shafarevich's theorem \cite{Neuk}: the group $G_m$ is solvable, having odd order, and therefore it is the Galois group of some extension $F_m$ of the global field $K(T)$. 
Note next that the groups $G_m$ are of prime-to-$p$ order. In particular branching is automatically tame and 
so the original and the tame versions of  {\rm (Ub-dec-d)} are equivalent. Proposition \ref{c_ex_p} concludes the proof. $\square$

\subsection{Bounding the branch point set} 
\label{ssec:alg-clos-car-positive}

\subsubsection{A second question} \label{ssec:refined-question}
Here we show that (Ub-dec-d) does not imply that $F\subset K(B)^{(d)}$ for some $d$ even if we assume further that the branch point set is finite. However the base field will
be algebraically closed in our counter-examples (and so the Tchebotarev property 
will not hold).

\begin{theorem}  \label{thm:2-examples}
In situation {\rm (a)} or {\rm (b)} below, there is an infinite Galois extension $F/K(B)$ 
satisfying {\rm (Ub-dec-d)} but  such that $F\not\subset K(B)^{(d)}$ for any $d$ and that is branched at only finitely many points:
\vskip 1mm

\noindent
{\rm (a)} $K$ is an algebraically closed field of characteristic $p>0$ and $B$ is a curve of genus
$\geq 1$.
\vskip 1mm

\noindent
{\rm (b)} $K$ is an algebraically closed field of characteristic $0$ and $B=\Pp^1$.
\end{theorem}

\subsubsection{Proof of case (a): fields of positive characteristic}
Assume that $K$ is an alge\-bra\-ical\-ly closed field of characteristic $p>0$ and $B$ is a curve of genus $g$. We use again the construction from \S \ref{ssec:construction}; we retain the notation from there. From proposition \ref{c_ex_p}, 
we are left with realizing all groups $G_m$ as groups of Galois extensions
$F_m/K(B)$ ($m\geq 1$) with controlled branching. 
We will 
use Abhyankar\rq{}s Conjecture on Galois groups of function field extensions of characteristic $p$, which was proved by the work of M.~Raynaud \cite{raynaud} and D.~Harbater \cite{harbater}:
\vskip 2mm

\noindent
{(The Raynaud-Harbater theorem)} {\it A finite group $G$ can be realized as the group of a Galois extension $F/K(B)$ unbranched outside a finite set $S$ if and only if the minimal number of generators of the quotient $G/p(G)$ of $G$ by the subgroup of $G$ generated by all $p$-Sylow subgroups of $G$ is at most $|S|+2g-1$.}

%
%
\vskip 2mm

Take $\ell = p$.  
For each $m\geq 1$, we have the following. The group $p(G_m)=\ell(G_m)$ is a normal subgroup of $G_m$ which properly contains the $p$-group $E_m$ (since $E_m$ is not normal in $G_m$). Consequently the group $p(G_m)\cap W_m$ is a non trivial normal subgroup 
of $G_m$. 
But as part of the theory of extraspecial groups, $W_m$ is a minimal non trivial normal subgroup of $G_m$. Therefore $W_m\subset p(G_m)$ so finally $p(G_m)=G_m$. 
From the Raynaud-Harbater theorem, if $g\geq 1$, then $G_m$ is the group of some Galois extension $F_m/K(B)$ 
unbranched everywhere. $\square$

%

\begin{remark} \label{rem:l=p-and-g2}
(a) For $g=0$, the construction leads to an extension $F/K(T)$ that is only branched at one point, say the point $\infty$. There is necessarily wild branching and proposition \ref{c_ex_p} guarantees that the decomposition degrees at all $t_0\in \Pp^1(K)\setminus\{\infty\}$ are uniformly bounded.
\vskip 1mm

\noindent
(b) We took $\ell = p$. If $\ell \not= p$, then if $q\not=p$, $p(G_m)$ is trivial and $G_m/p(G_m)=G_m$  and,  if $q=p$, $p(G_m)=q(G_m)=W_m$. So $G_m/p(G_m)$ cannot be generated by less than $2m$ generators ($E_m$ is of rank $2m$) and $G_m$ cannot be realized with branch points in a fixed finite set $S$.
\end{remark}

\subsubsection{Proof of case (b): fields of characteristic $0$} \label{ssec:K=Qbar}

Assume that $K$ is an algebraically closed field of characteristic $0$ and $B=\Pp^1$.

Fix an odd prime $p$. For each $m\geq 1$, take for $G_m$ the dihedral group $\mathbb{Z}/p^m\mathbb{Z}\rtimes \mathbb{Z}/2\mathbb{Z}$ of order $2p^m$. The projective limit $G=\limproj_{m\geq 1} G_m$ is the pro-dihedral group $\Zz_p \rtimes \Zz/2\Zz$. Denote by $C_m$  (resp. $C$) the conjugacy class of $G_m$ (resp. of $G$) of all elements $(x,1)$ with $x\in \Zz/p^m\Zz$ (resp. with $x\in \Zz_p$). These are conjugacy classes of elements of order $2$.

Pick two elements  $\sigma, \tau\in C$ and denote by $\sigma_m$ and $\tau_m$ their images in $C_m$ {\it via} the projection map $G\rightarrow G_m$. We have $\sigma_m \sigma_m \tau_m \tau_m=1$ and $G_m=\langle \sigma_m,\tau_m\rangle$ ($m\geq 1$). By the Riemann existence theorem, if we choose four distinct points $t_1, t_2, t_3, t_4\in \mathbb{P}^1(K)$, there is a Galois extension $F_m/K(T)$, with group $G_m$, branch points $t_1, t_2, t_3, t_4$ and corresponding inertia groups $\langle \sigma_n\rangle$ and its conjugates for $t_1,t_2$ and $\langle\tau_n\rangle$ and its conjugates for 
$t_3,t_4$ ($m\geq 1$). Furthermore, by a classical compactness argument based on the fact that for each $m\geq 1$ and each $4$-tuple $(t_1, t_2, t_3, t_4)$ as above, there are only finitely many choices of the extension $F_m/K(T)$, one can perform  the construction compatibly, {\it i.e.}, so that $F_m/K(T)$ is  obtained from $F_{m+1}/K(T)$ {\it via} the epimorphism $G_{m+1}\rightarrow G_m$ ($m\geq 1$).

Set $F=\limind_{m\geq 1} F_m$. The extension $F/K(T)$ is Galois of group $G$. 
%
%
%
%
%
For each $m\geq 1$, the exponent of $G_m$ is $\geq p^m$ and so $G$ is not of
finite exponent. As already noticed (\S  \ref{ssec:refinedquestion}), this implies that $F$ cannot be a subfield of $K(B)^{(d)}$ for any $d$.
As $K$ is algebraically closed,  for each $t_0\in \Pp^1(k_v)$, the {local decomposition degree} at  $t_0$ is the branching index. By construction, it is $1$ or $2$. So condition {\rm (Ub-dec-d)} holds. $\square$

\subsection{Three final remarks}

The following three remarks relate to case (b) of theorem \ref{thm:2-examples}. 
As in this statement assume that $K$ is an algebraically closed field of characteristic $0$.

\subsubsection{On the geometric Bogomolov property} \label{sec:bogomolov} 
Consider the (smooth pro\-jective) curves $C_m$ corresponding to the function fields $F_m$ from the proof above ($m\geq 1$). The degrees $[F_m:K(T)]$ go to infinity and the Riemann-Hurwitz formula shows that the curves $C_m$ are all of genus $1$. 

We explain below that this provides a counter-example to a geometric analog of a result of Bombieri and Zannier around the Bogomolov property. The ``geometric Bogomolov property'' as presented below is stated by J. Ellenberg in \cite{Eblog}.

Recall that the \emph{gonality} of some $K$-curve $C$ is the least degree of a non constant function $x\in K(C)$ and that  the gonality of a curve is bounded above in terms of its genus. Consequently in our example above, we have that there is no real constant $c>0$ such that
\vskip 2mm

\noindent
{\rm (GB)}  {\it the gonality of $C_m$ is $\geq c\hskip 2pt [F_m:K(T)]$ ($m\geq 1$).}

\vskip 2mm

Condition (GB) can be rewritten in terms of the absolute logarithmic height on $\overline{K(T)}$. Given a non constant function $x\in \overline{K(T)}$, the absolute logarithmic height of $x$, denoted by $h(x)$, is defined as follows:  if $L/K(T)$ is any finite extension such that $x\in L$, $h(x)$ is the ratio $[L:K(x)]/[L:K(T)]$. Noting that if $C$ is a curve corresponding to the function field $L$, then $[L:K(x)]$ is the degree of $x$ on $C$ (equivalently, the number of zeroes (or poles) on $C$), condition (GB) rewrites:
\vskip 2mm

\noindent
{\rm (GB)}  {\it for every non constant function $x$ in $F$, $h(x)\geq c$.}

\vskip 2mm
In \cite{Eblog}, J.~Ellenberg says that an infinite algebraic extension $F/K(T)$ has the {\it geometric Bogomolov property} if there exists some $c>0$ such condition (GB) holds. 
This is his geometric analog of the {\it Bogomolov property} of an algebraic 
extension $F/\overline \Qq$ (introduced in \cite{BZ}), which requests that there exists some $c>0$ such that if $x\in F$ is neither zero nor a root of unity, then $h(x)\geq c$, where $h(x)$ is the classical Weil logarithmic height on $\overline \Qq$.

For the Bogomolov property of algebraic extensions $F/\Qq$, we have the following criterion due to Bombieri-Zannier \cite[theorem 2]{BZ}, which has several interesting consequences (for example that the 
field $\mathbb{Q}^{tp}$ of totally $p$-adic numbers has the Bogomolov property, just as the field $\mathbb{Q}^{tr}$ 
of totally reals does (a result of Schinzel \cite{Schi})).

%
%


\vskip 2mm

\noindent
{(Bombieri-Zannier criterion)} {\it If $F/\mathbb{Q}$ is an algebraic extension with finite local degrees at some prime $p$, then $F$ has the Bogomolov property. 
}

%
%
\vskip 2mm


Our original example --- an infinite extension $F/K(T)$ which has uniformly 
bounded local decomposition degrees  (here they are just the ramification 
indices) but does not satisfy property {\rm (GB)}  --- shows that the geometric 
analog of the Bombieri-Zannier criterion does not hold, even if {\it all} 
decomposition degrees are assumed to be bounded (and not just 
the local degrees above {\it one} prime).

\subsubsection{A generalization using universal $p$-Frattini covers} 
The construction from \S \ref{ssec:K=Qbar}
extends to the following more general context; we refer to  \cite{Fr_introMT}, 
\cite[\S 22]{FrJa}, \cite{DeLy04}, for details. 

A group $G_1$ is given with a prime $p$ such that $p| \hskip 1mm |G_1|$ and $G_1$ is $p$-perfect, {\it i.e.} $G_1$ is generated by its elements of prime-to-$p$ order. Take for $G$ the {\it $p$-universal Frattini cover of $G_1$} (which generalizes the pro-dihedral group $\Zz_p\rtimes \Zz/2\Zz$) and for $(G_m)_{m\geq 1}$ the natural collection of finite characteristic quotients of $G$ (which generalize the dihedral groups $\Zz/p^m \Zz \rtimes \Zz/2\Zz$, $m\geq 1$). Select $r$ elements of $G_1$
of prime-to-$p$ order generating $G_1$.  The conjugacy class of each of these elements
can be lifted to a conjugacy class $C_i$ of $G$ with the same order, $i=1,\ldots,r$ ({the lifting lemma}).
Pick an element $\sigma_i\in C_i$, $i=1,\ldots r$ and consider the $2r$-tuple $(\sigma_1,\sigma_1^{-1}\ldots,\sigma_r,\sigma_r^{-1})$; its entries generate $G$ (the Frattini property) and are of product one.

Extensions $F_m/K(T)$ can then be constructed as in \S \ref{ssec:K=Qbar} with the $2r$-tuple above replacing the $4$-tuple $(\sigma,\sigma,\tau,\tau)$ and $2r$ distinct points of $\Pp^1(K)$ replacing the $4$ chosen points $t_1,\ldots,t_4\in \Pp^1(K)$ in \S \ref{ssec:K=Qbar}. Set $F=\limind_{m\geq 1} F_m$. The extension $F/K(T)$ is Galois of group $G$ and it satisfies {\rm (Ub-dec-d)} but is not  contained in  $K(B)^{(d)}$ for any $d$.
The main point is that $G$ is still of infinite exponent in this more general context. Indeed the $p$-Sylow subgroups of $G$ are known to be free pro-$p$ groups and so cannot have non trivial elements of finite order. 
%
%
%
%

\subsubsection{In the abelian situation {\rm the following can be added}}

\begin{proposition}
Let $F/K(T)$ be an abelian extension, with finitely many branch points and such that condition {\rm (Ub-dec-d)} holds. Then not only $F\subset K(B)^{(d)}$ but $F/K(T)$ is finite. 
\end{proposition}
 
\begin{proof} Denote the branch points of $F/K(T)$ by
$t_1,\ldots,t_r$. Let $F_0/K(T)$ be a finite Galois sub-extension of $F/K(T)$ of group $G_0$.
From the Riemann existence theorem, $G_0$ is generated by $r$ elements $\sigma_1,\ldots,\sigma_r$ such that $\sigma_1 \cdots\sigma_r = 1$; moreover $\sigma_i$ is a generator 
of some inertia group above $t_i$. From assumption (Ub-dec-d), the order of $\sigma_i$ is bounded by some constant $\delta$, independent of $i$. Since $G_0$ is abelian we have $|G_0|\leq {\delta}^{r-1}$. 
As all finite sub-extensions of $F/K(T)$ are abelian and the argument holds for any of them, conclude that $F/K(T)$ is finite and that $[F:K(T)]\leq {\delta}^{r-1}$.
\end{proof}

\bibliography{LocGlob8}
\bibliographystyle{alpha}

\end{document}